\documentclass[12pt,a4paper,reqno]{amsart} 
\usepackage{amssymb}
\usepackage{latexsym}
\usepackage{amsmath}
\usepackage{mathrsfs}
\usepackage{calc}                         

\newcommand{\rr}[1]{\mathbf R^{#1}}

\newcommand{\nm}[2]{\Vert #1\Vert _{#2}}

\newcommand{\sets}[2]{\{ {\,}#1{\,};{\,}#2{\,}\} }

\newcommand{\cdo}{\, \cdot \, }

\newcommand{\supp}{\operatorname{supp}}

\newcommand{\eabs}[1]{\langle #1\rangle}

\newcommand{\vrum}{\vspace{0.1cm}}
\newcommand{\FL}{\mathscr F \! L}
\newcommand{\masfR}{\mathsf R}
\DeclareMathOperator{\WF}{WF}

\numberwithin{equation}{section}          

\newtheorem{thm}{Theorem}
\numberwithin{thm}{section}
\newtheorem{prop}[thm]{Proposition}

\newtheorem{lemma}[thm]{Lemma}

\theoremstyle{definition}

\theoremstyle{remark}

\title{\textbf {A note on products in weighted Fourier-Lebesgue spaces}}
\author{Karoline Johansson}

\address{Department of Computer science, Mathematics and Physics,
Linn{\ae}us University, V{\"a}xj{\"o}, Sweden
}

\email{karoline.johansson@lnu.se}

\author{Stevan Pilipovi\' c}

\address{Department of Mathematics and Informatics,
University of Novi Sad, Novi Sad, Serbia}

\email{stevan.pilipovic@dmi.uns.ac.rs}

\author{Nenad Teofanov}

\address{Department of Mathematics and Informatics,
University of Novi Sad, Novi Sad, Serbia}

\email{nenad.teofanov@dmi.uns.ac.rs}

\author{Joachim Toft}

\address{Department of Computer science, Mathematics and Physics,
Linn{\ae}us University, V{\"a}xj{\"o}, Sweden
}

\email{joachim.toft@lnu.se}

\begin{document}

\begin{abstract}
We consider multiplication properties of elements in weighted
Fourier Lebesgue and modulation spaces. Especially we extend some results
in  \cite{PTT2}.
\end{abstract}

\maketitle

\par

\section{Introduction}\label{sec0}

\par

In this paper we extend some results from  \cite{PTT2} concerning multiplication
properties in Fourier-Lebesgue and modulation spaces.

\par

One of the goals is to estimate the parameters $s$ and $q$ such that
$ f_1 f_2  \in \mathscr F L^q _s $ if $f_j \in  \mathscr F L^{q_j} _{s_j}$, $ j = 1,2.$
This is done in Theorem \ref{together}.
Just to give a flavor of our results, we give below a special interesting
case when $ q_1 $ or $ q_2 $ is greater than $2$. Here and in what follows it is convenient to
consider the functional
\begin{equation}\label{Rqfunctional}
\masfR (q)\equiv 2-\frac 1{q_0}-\frac 1{q_1}-\frac 1{q_2},\qquad q=(q_0,q_1,q_2)\in [1,\infty]^3.
\end{equation}

\par

\begin{prop}\label{prop0}
Let $ 0 \leq s_j + s_k$, $j\neq k$, $\masfR (q)$ be as in
\eqref{Rqfunctional}, and let
$f_j \in  \mathscr F L^{q_j} _{s_j}$, $ j = 1,2$. If
$$
0 \le \masfR (q)\le \frac 12
\quad \text{and}\quad
0\leq s_0+s_1 + s_2-  d\cdot  \masfR (q),
$$
with the strict inequality when $\masfR (q)>0$ and $s_j=d\cdot  \masfR (q) $
for some $j=0,1,2$, then  $f_1 f_2  \in \mathscr F L^{q_0'} _{-s_0}$.
\end{prop}

\par

We note that Proposition \ref{prop0} is a special case Theorem \ref{together} below.
Moreover, by letting $ q_1 = q_2 = q _0= 2 $, 
Proposition \ref{prop0} agrees with
the H\"ormander theorem on microlocal regularity of a product
\cite[Theorem 8.3.1]{Hrm-nonlin}.

\par

From Theorem \ref{together} below it also follows that Proposition \ref{prop0}
remains true after the Fourier Lebesgue spaces $\mathscr FL_{s_j}^{q_j}$ have
been replaced by the modulation or Wiener amalgam spaces $M^{p_j,q_j}_{s_j}$
and $W^{p_j,q_j}_{s_j}$, respectively, when
\begin{equation}\label{pjHoldercond}
\frac 1{p_0} + \frac 1{p_1} + \frac 1{p_2} = 1.
\end{equation}

\par

\subsection{Basic notions and notation} \label{notation}

\par

In this subsection we collect some notation and notions which will be used in the sequel.

\par

We put $\mathbf N =\{0,1,2,\dots \}$, $\eabs x =(1+|x|^2)^{1/2}$,
for $x\in \rr d$, and $A\lesssim B$ to
indicate $A\leq c B$ for a suitable constant $c>0$.
The scalar product in $L^2 $ is denoted  by
$(\cdo  , \cdo )_{L^2} = (\cdo  , \cdo )$.

\par

\section{Main results}\label{sec1}

\par

In this section we extend some results from  \cite{PTT2}.
%
%
%
%
Our main main result is Theorem \ref{together}, where we present sufficient conditions on
$s_j\in \mathbf R$ and $q_j\in [1,\infty]$, $j=0,1,2$, to ensure that $f_1f_2\in
\mathscr F L^{q_0} _{s_0} $ when $f_j \in  \mathscr F L^{q_j} _{s_j}$,
$ j = 1,2$. The result also include related multiplication properties for
modulation and Wiener amalgam spaces.

\par

%
%
%
%
%
%
%

\par

Let $\phi \in \mathscr S(\rr d)\setminus 0$, $s,t\in \mathbf R$ and $p,q\in [1,\infty]$  be fixed. We
recall that the modulation space $M^{p,q}_{s,t}(\rr d)$ consists of all $f\in \mathscr S'(\rr d)$
such that
$$
\nm f{M^{p,q}_{s,t}}\equiv \left ( \int _{\rr d} \left ( \int _{\rr d} |V_\phi f(x,\xi )\eabs x^t
\eabs \xi ^s|^p\, dx  \right )^{q/p}d\xi  \right )^{1/q}
$$
is finite (with obvious interpretation of the integrals when $p=\infty$ or $q=\infty$). In the same
way, the modulation space $W^{p,q}_{s,t}(\rr d)$ consists of all $f\in \mathscr S'(\rr d)$
such that
$$
\nm f{W^{p,q}_{s,t}}\equiv \left ( \int _{\rr d} \left ( \int _{\rr d} |V_\phi f(x,\xi )\eabs x^t
\eabs \xi ^s|^q\, d\xi  \right )^{p/q}dx  \right )^{1/p}
$$
is finite.

\par

\begin{lemma}\label{1/2}
Assume that $x_j= 1/q_j$. If $0 \le x_j \le 1$, then 
$$
\masfR(q)= 2-\sum _{j=0}^2 x _j
$$
and the following statements are equivalent  
\begin{equation}\label{lastineq1}
0 \le \masfR (q) \le  \frac{1}{2}
\end{equation}
and
\begin{equation}\tag*{(\ref{lastineq1})$'$}
0 \le \masfR (q) \le \max \left ( \frac{1}{2}, \min \left (
\frac {1}{q_0}, \frac {1}{q_1},\frac {1}{q_2} \right ) \right ).
\end{equation}
\end{lemma}

\par

\begin{proof}
It is obvious that \eqref{lastineq1} implies \eqref{lastineq1}$'$ . Next assume that \eqref{lastineq1}$'$  holds. If $\masfR(q)>1/2$, then $\min x_j>1/2$, which implies that  
$$
\masfR(q)= 2-\sum _{j=0}^2 x _j <2-\frac{3}{2} = \frac{1}{2}.
$$
Since this is a contradition, it follows that $\masfR(q)\le 1/2$ and the inequality \eqref{lastineq1} holds.
\end{proof}

\par

\begin{thm} \label{together}
Let $X\subseteq \rr r$ be open, $s_j,t_j \in \mathbf R$, $ p_j,q_j \in
[1,\infty] $, $j=0,1,2$, and let $\masfR (q)$ be as in \eqref{Rqfunctional} and satisfy \eqref{lastineq1} or \eqref{lastineq1}$'$.
Also assume that \eqref{pjHoldercond}
\begin{align}
%
%
%
%
%
&\begin{aligned}\label{lastineq2}
 0&\leq s_j+s_k,  \quad j,k=0,1,2,  \quad j\neq k, \quad
 \mbox{and} \\[1ex]
0 &\leq s_0+s_1 + s_2-  d \cdot  \masfR (q),
\end{aligned}
\end{align}
hold, with strict inequality in the last inequality in \eqref{lastineq2}
when $\masfR (q)>0$ and $s_j=d\cdot  \masfR (q) $
for some $j=0,1,2$.

 Then the following is true:
\begin{enumerate}
\item the map $(f_1,f_2)\mapsto f_1f_2$ on $C_0^\infty (\rr d)$
extends uniquely to a continuous map from $ \mathscr F L^{q_1} _{s_1}(\rr d)
\times  \mathscr F L^{q_2} _{s_2}(\rr d)$ to $\mathscr F L^{q_0'} _{-s_0}(\rr d)$;

\vrum
\item the map $(f_1,f_2)\mapsto f_1f_2$ on $C_0^\infty (X)$
extends uniquely to a continuous map from $ (\mathscr F L^{q_1} _{s_1})_{loc}(X)
\times  (\mathscr F L^{q_2} _{s_2})_{loc}(X)$ to
$(\mathscr F L^{q_0'} _{-s_0})_{loc}(X)$;

\vrum

\item if $0\le t_0+t_1+t_2$, then the map $(f_1,f_2)\mapsto f_1f_2$ on $C_0^\infty (\rr d)$
extends to a continuous map from $M^{p_1,q_1} _{s_1,t_1}(\rr d)
\times  M^{p_2,q_2} _{s_2,t_2}(\rr d)$ to $M^{p_0',q_0'} _{-s_0,-t_0}(\rr d)$. The extension is
unique when $p_j,q_j<\infty$, $j=1,2$;

\vrum

\item if $t_0\le t_1+t_2$, then the map $(f_1,f_2)\mapsto f_1f_2$ on $C_0^\infty (\rr d)$
extends to a continuous map from $W^{p_1,q_1} _{s_1,t_1}(\rr d)
\times  W^{p_2,q_2} _{s_2,t_2}(\rr d)$ to $W^{p_0',q_0'} _{-s_0,-t_0}(\rr d)$. The extension is
unique when $p_j,q_j<\infty$, $j=1,2$.
\end{enumerate}

\end{thm}

\par

Next we  apply the above result to
estimate the wave-front set of products of functions from different
Fourier-Lebesgue spaces. This is an extension of \cite[Theorem 8.3.3 (iii)]{Hrm-nonlin},
see also \cite[Theorem 4.3]{PTT2}.

\par

\begin{thm}\label{microlocal-functions-some-more}
Let $s_j \in \rr d $, $ q_j \in [1,\infty] $, $j=0,1,2$, and let $\masfR (q)$ in
\eqref{Rqfunctional} be such that
\eqref{lastineq1} and \eqref{lastineq2} hold with strict inequality in the last
inequality in \eqref{lastineq2} when $s_0$, $s_1$ or $s_2 $ or $-s_0 $ is equal to
$ d \cdot  \masfR (q) $.
If $f_j \in \big ( \mathscr F L^{q_j} _{s_j}\big )_{loc} (X)$, $j=1,2$, then
$f_1f_2$ is well-defined as an element in $\mathscr D'(\rr d)$, and
$$
\WF _{\mathscr F L^{q_0'} _{-s_0}} (f_1 f_2) \subseteq
\WF _{\mathscr F L^{q_{1}} _{s_1}} (f_1)\cup \WF _{\mathscr F L^{q_2} _{s_2}} (f_2).
$$
\end{thm}

\par

\section{The map $T_F(f,g)$} \label{sec2}

\par

In this Section we introduce and study a convenient bilinear map (denoted by 
$T_F $ here below when $F\in L^1_{loc}$ is appropriate).

\par

For $F \in L^1 _{loc} (\rr {2d}) $ and $ p,q \in [1,\infty] $, we
set
\begin{align*}
\nm F{L^{p,q}_1} &\equiv \Big ( \int \Big ( \int |F(\xi ,\eta )|^p\,
d\xi \Big )^{q/p}\, d\eta  \Big )^{1/q}
\intertext{and} \nm
F{L^{p,q}_2} &\equiv \Big ( \int \Big ( \int |F(\xi ,\eta )|^q\,
d\eta \Big )^{p/q}\, d\xi  \Big )^{1/p},
\end{align*}
and we let $L^{p,q}_1(\rr {2d})$ be the set of all $F \in L^1 _{loc} (\rr {2d})$ such that
$\nm F{L^{p,q}_1}$ is finite. The space $L^{p,q}_2$ is defined analogously.  (Cf.
\cite{PTT, PTT2}.) We also let $\Theta$ be defined as
\begin{equation}\label{FGrel}
(\Theta F)(\xi ,\eta ) = F(\xi ,\xi-\eta ),\qquad F\in L^1_{loc}(\rr {2d}).
\end{equation}

\par

If $F\in L^1_{loc}(\rr {2d})$ is fixed, then we are especially concerned about
extensions of the mappings
\begin{alignat}{2}
(F,f,g) &\mapsto & T_F(f,g) &\equiv \int F(\cdo  ,\eta )f(\eta )g(\cdo  -\eta )\, d\eta \label{TFmap}
\intertext{and}
(F,f,g) &\mapsto & T_{\Theta F}(f,g) &\equiv \int F(\cdo  ,\eta )f(\cdo  - \eta )g(\eta )\, d\eta .\label{TGmap}
\end{alignat}
from $ C^\infty _0 (\rr d) \times  C^\infty _0 (\rr d)$
to $  \mathscr S' (\rr d)$.

\par

The following extend \cite[Lemma 8.3.2]{Hrm-nonlin} and \cite[Proposition
3.2]{PTT2}.

\par

\begin{prop}\label{prop1}
Let $F\in L^1_{loc}(\rr {2d})$, $q_j\in [1,\infty
]$, $j=0,1,2$. Also assume that $\masfR (q)$ in \eqref{Rqfunctional} is non-negative,
and let $r=1/\masfR (q)\in (0,\infty ]$.
Then the following is true:
\begin{enumerate}
\item if $\masfR (q) \leq 1/q_0'$, then the mappings \eqref{TFmap}
and \eqref{TGmap} are continuous
from $L^{\infty ,r}_2(\rr {2d})\times L^{q_1}(\rr d)\times
L^{q_2}(\rr d)$ to $L^{q_0}(\rr d)$. Furthermore,
\begin{align}
\nm{T_F(f,g)}{L^{q_0}} &\lesssim
\nm{F}{L^{\infty,r}_2}\nm{f}{L^{q_1}}\nm{g}{L^{q_2}}\label{TFL2est1}
\intertext{and}
\nm{T_{\Theta F}(f,g)}{L^{q_0}} &\lesssim
\nm{F}{L^{\infty,r}_2}\nm{f}{L^{q_1}}\nm{g}{L^{q_2}}.
\label{TFL2est2}
\end{align}

\vrum

\item if in addition $\masfR (q)\le \max (1/2,1/q_1)$,
then the map \eqref{TFmap} is continuous from $L^{r,\infty}_1(\rr
{2d})\times L^{q_1}(\rr d)\times L^{q_2}(\rr d)$ to $L^{q_0}(\rr d)$.
Furthermore,
\[
\nm{T_F(f,g)}{L^{q_0}} \lesssim \nm{F}{L^{r,\infty}_1}
\nm{f}{L^{q_1}}\nm{g}{L^{q_2}}.
\]

\vrum

\item if in addition $\masfR (q)\le \max (1/2,1/q_2)$,
then the map \eqref{TGmap} is continuous from $L^{r,\infty}_1(\rr
{2d})\times L^{q_1}(\rr d)\times L^{q_2}(\rr d)$ to $L^{q_0}(\rr d)$.
Furthermore, 
\[
\nm{T_{\Theta F}(f,g)}{L^{q_0}} \lesssim \nm{F}{L^{r,\infty}_1}
\nm{f}{L^{q_1}}\nm{g}{L^{q_2}}.
\]
\end{enumerate}
\end{prop}

\par

We note that Proposition \ref{prop1} agrees with \cite[Lemma
8.3.2]{Hrm-nonlin} when $ q_1 = q_2 = 2 $ and with \cite[Proposition
3.2]{PTT2} when $ q_1 = q_2 \in [1,\infty] $.

\par

\begin{proof}
(1) We only prove \eqref{TFL2est1} and leave $\eqref{TFL2est2}$ for the reader.

\par

First, assume that $q_1,q_2<\infty$, and let $f,g\in C_0^\infty (\rr d)$. By
H{\"o}lder's inequality we get
\begin{multline}
\Big(\int |T_F(f,g)(\xi)|^{q_0} \, d \xi\Big)^{1/q_0}
\\[1 ex]
\leq \Big ( \int \Big [ \Big ( \int |F(\xi ,\eta )|^{r} \, d\eta
\Big )^{1/r} \Big ( \int | f(\eta)|^{r'} |g(\xi -\eta
)|^{r'}\, d\eta \Big) ^{1/r'} \Big ] ^{q_0}\, d\xi \Big ) ^{1/q_0}.
\end{multline}
Next we use
$ r \geq q_0'$ and Young's inequality to obtain
\begin{multline}\label{TFcomputations}
\Big(\int |T_F(f,g)(\xi)|^{q_0} \, d \xi\Big)^{1/q_0}
\\[1 ex]
\leq \nm {F}{L^{\infty ,q_0}_2} \Big ( \nm
{|f|^{r'}*|g|^{r'}}{L^{q_0/r'}} \Big )^{1/r'} \leq \nm
{F}{L^{\infty,r}_2} \Big ( \nm {|f|^{r'}}{L^{r_1}} \nm
{|g|^{r'}}{L^{r_2}} \Big )^{1/r'}
\\[1 ex]
= \nm {F}{L^{\infty,r}_2} \nm{f}{L^{q_1}}\nm{g}{L^{q_2}},
\end{multline}
where $ r_1 = q_1/r' $ and $r_2 = q_2/r'$.
The result now follows from the fact that $C^\infty_0$ is dense in
$L^{q_1}$ and $L^{q_2}$ when $q_1, q_2 < \infty$.

\par

Next, assume that $q_1=\infty$ and $q_2<\infty$, and let $f\in
L^\infty$ and $g\in C_0^\infty$. Then, it follows that $T_F(f,g)$ is
well-defined, and that \eqref{TFcomputations} still holds. The
result now follows from the fact that $C_0^\infty$ is dense in
$L^{q_2}$. The case
$q_1<\infty$ and $q_2=\infty$ follows analogously.

\par

Finally, if $q_1=q_2=\infty$, then the assumptions
implies that $r=1$ and $q_0=\infty$. The inequalities
\eqref{TFL2est1} and \eqref{TFL2est2} then follow  by H{\"o}lder's inequality.

\par

(2) First we consider the case $r\ge q_1$. Let $h\in C_0(\rr d)$
when $r <\infty$ and $h\in L^1(\rr d)$ if $r=\infty$. Also let
$F\in L^{r,\infty}_1(\rr {2d})$ and $F_0(\eta,\xi)=F(\xi,\eta)$
and $\check g(\xi)=g(-\xi)$. By \cite[page 354]{PTT2}, we have
$ |\left \langle T_F (f,g) , h \right \rangle | = | \left \langle
T_{F_0}(h,\check g),f \right \rangle | $.
Then (1) implies
\begin{multline*}
|\left \langle T_F (f,g) , h \right \rangle | = | \left \langle
T_{F_0}(h,\check g),f \right \rangle |
\\[1 ex]
\leq \nm {T_{F_0}(h, \check g)}{L^{q_1'}} \nm {f}{L^{q_1}} \leq
\nm{F_0}{L_2^{\infty ,r}} \nm {f}{L^{q_1}} \nm {h}{L^{q'}}
\nm{g}{L^{q_2}}
\\[1 ex]
\leq
\nm{F}{L_1^{r,\infty}} \nm {f}{L^{q_1}} \nm {h}{L^{q'}}
\nm{g}{L^{q_2}}.
\end{multline*}

\par

Next, assume that $r \ge 2$ and $F\in L^{r,\infty}_1(\rr {2d})$.
We will prove the assertion by interpolation.
First we consider the case $r=\infty$. Then
$ \masfR (q) = 0$, and
\begin{multline*}
\Big \Vert \int F(\xi,\eta)f(\eta)g(\xi-\eta)\, d\eta \Big \Vert
_{L^{q_0}} \leq \nm{F}{L^{\infty,\infty}_1} \nm{|f|*|g|}{L^{q_0}}
\\[1 ex]
\leq
\nm{F}{L^{\infty,\infty}_1}\nm{f}{L^{q_1}}\nm{g}{L^{q_2}}.
\end{multline*}

\par

For the case $r=2$ we have $\masfR (q)=1/2$.
By letting
$$
M =\nm {F}{L^{2,\infty}_1}, \quad  \theta = \frac {( \nm {g}{L^{2r_1}}
\nm{h}{L^{2r_2}} )^{1/q_1} }{ \nm {f}{L^{q_1}}^{1/q_1'}},\quad
r_1=q_2/2\quad \text{and} \quad r_2=q_0'/2,
$$
it follows from Cauchy-Schwartz inequality, the weighted
arithmetic-geometric mean-value inequality and Young's inequality that
\begin{multline*}
|\left \langle T_F(f,g),h \right \rangle |
%
%
\leq \int \Big ( \int |F(\xi ,\eta )| |g(\xi -\eta )| |h(\xi )|\,
d\xi \Big ) |f(\eta )| \, d\eta
\\[1 ex]
\le M \int \Big ( \int |g(\xi - \eta) |^2 |h(\xi )|^2 \, d\xi \Big )
^{1/2} |f(\eta )| \, d \eta
\end{multline*}
\\[-4ex]
$$
\leq M \int \Big ( \frac{\theta^{q_1}}{q_1} |f(\eta)|^{q_1} +
\frac{1}{q_1' \theta^{q_1'}} \Big ( \int |g(\xi -\eta )|^2
|h(\xi )|^2 \, d\xi \Big ) ^{q_1'/2} \Big ) \, d\eta
$$
\begin{multline*}
=  M \Big ( \frac{\theta^{q_1}}{q_1} \nm{f}{L^{q_1}}^{q_1} +
\frac{1}{q_1' \theta^{q_1'}} \nm {|g|^2 * |h| ^2}
{L^{q_1'/2}}^{q_1'/2} \Big )
\\[1 ex]
\leq M \Big ( \frac {\theta ^{q_1}}{q_1} \nm {f}{L^{q_1}}^{q_1} +
\frac{1}{q_1' \theta ^{q_1'}} ( \nm {|g|^2}{L^{r_1}} \nm
{|h|^2}{L^{r_2}})^{q_1'/2} \Big )
\\[1 ex]
= M \Big ( \frac { \theta^{q_1}}{q_1} \nm {f}{L^{q_1}}^{q_1} +
\frac {1}{q_1' \theta ^{q_1'}} (\nm {g} {L^{2r_1} } \nm
{h}{L^{2r_2}}) ^{q_1'} \Big)
\\[1 ex]
=  M \Big (\frac {1}{q_1}+\frac {1}{q_1'} \Big ) \nm {f}{L^{q_1}}\nm
{g}{L^{q_2}} \nm {h}{L^{q_0'}}
\\[1 ex]
=  M \nm {f}{L^{q_1}} \nm {g}{L^{q_2}} \nm{h}{L^{q_0'}}.
\end{multline*}
This gives the result for $r=2$.

\par

Since we also have proved the result for $r=\infty$. The
assertion (2) now follows for general $r\in [2,\infty ]$ by
multi-linear interpolation, using Theorems 4.4.1, 5.1.1 and 5.1.2 in \cite{BL}.
%

\par

The assertion (3) follows by similar arguments as in the proof of
(2). The details are left for the reader. The proof is complete.

\end{proof}

\section{Proof of Theorems \ref{together} and
\ref{microlocal-functions-some-more}} \label{sec3}

\par

Before the proof of Theorem \ref{together}, we need some preparation,
and formulate auxiliary results in three Lemmas.

\par

First, we recall \cite[Lemma 3.5]{PTT2} which concerns
different integrals of the function
\begin{equation}\label{Fdef}
F(\xi ,\eta ) = \eabs \xi ^{s_0}\eabs {\xi -\eta}^{-s_1}\eabs \eta
^{-s_2}, \;\;\; \xi ,\eta \in \rr d,
\end{equation}
where $s_j\in \mathbf R$, $j=0,1,2$. These integrals, with respect to $\xi$ or $\eta$, are taken
over the sets
\begin{equation}\label{theomegasets}
\begin{aligned}
\Omega _1 &= \sets{ (\xi, \eta) \in \rr {2d}}{\eabs \eta < \delta
\eabs \xi },
\\[1ex]
\Omega _2 &= \sets{ (\xi, \eta) \in \rr {2d}}{\eabs {\xi -\eta }<
\delta \eabs \xi },
\\[1ex]
\Omega _3 &= \sets{ (\xi, \eta) \in \rr {2d}}{\delta \eabs \xi \le
\min (\eabs \eta ,\eabs {\xi -\eta }),\ |\xi |\le R},
\\[1ex]
\Omega _4 &= \sets{ (\xi, \eta) \in \rr {2d}}{\delta \eabs \xi
  \le \eabs {\xi -\eta }\le  \eabs \eta,\ |\xi |> R},
\\[1ex]
\Omega _5 &= \sets{ (\xi, \eta) \in \rr {2d}}{\delta \eabs \xi \le
  \eabs \eta \le \eabs {\xi -\eta },\ |\xi |> R},
\end{aligned}
\end{equation}
for some positive constants $\delta$ and $R$.
By $ \chi _{\Omega _j} $ we denote the characteristic function of the set $ \Omega_j,$ $ j = 1,\dots, 5.$

\par

\begin{lemma}\label{intestimates}
Let $F$ be given by \eqref{Fdef} and let $\Omega _1,\dots,\Omega _5$
be given by \eqref{theomegasets}, for some
constants $0<\delta <1$ and $R\ge 4/\delta$. Also let $p\in [1,\infty ]$
and $ F_j=\chi _{\Omega _j}F$, $ j = 1,\dots, 5.$ Then 
the following is true:
\begin{enumerate}
\item
$$
\nm {F_1(\xi ,\cdo  )}{L^p}\lesssim
\begin{cases}
\eabs \xi ^{s_0-s_1}\big ( 1+\eabs \xi
^{-s_2+d/p} \big ),& s_2\neq d/p ,
\\[1ex]
\eabs \xi ^{s_0-s_1}\big ( 1+\log
\eabs \xi  \big )^{1/p},& s_2= d/p \text;
\end{cases}
$$

\vrum

\item
$$
\nm {F_2(\xi ,\cdo  )}{L^p}\lesssim
\begin{cases}
\eabs \xi ^{s_0-s_2}\big ( 1+\eabs \xi
^{-s_1+d/p} \big ),& s_1\neq d/p ,
\\[1ex]
\eabs \xi ^{s_0-s_2}\big ( 1+\log
\eabs \xi  \big )^{1/p},& s_1= d/p \text;
\end{cases}
$$

\vrum

\item $\nm {F_3(\cdo ,\eta )}{L^p}\lesssim \eabs \eta ^{-s_1-s_2}$;

\vrum

\item if $j=4$ or $j=5$, then
$$
\nm {F_j(\cdo ,\eta )}{L^p}\lesssim
\begin{cases}
\eabs \eta ^{s_0-s_1-s_2+d/p},& s_0 > -d/p ,
\\[1ex]
\eabs \eta ^{-s_1-s_2}\big ( 1+\log
\eabs \eta  \big )^{1/p},& s_0 = -d/p,
\\[1ex]
\eabs \eta ^{-s_1-s_2},& s_0 < -d/p .
\end{cases}
$$
\end{enumerate}
\end{lemma}

\par

We refer to \cite{PTT2} for the proof of Lemma \ref{intestimates}.

\par

Next we estimate each of the auxiliary
functions $ T_{F_j}, $
defined by \eqref{TFmap} with $F$ replaced by $ F_j $, $ j = 1,\dots, 5$.

\par

\begin{lemma} \label{estimatesTFj}
Let $\masfR (q)$ and $F$ be given by \eqref{Rqfunctional} and \eqref{Fdef}, and let
$\Omega _1,\dots,\Omega _5$ be given by \eqref{theomegasets}, for some
constants $0<\delta <1$ and $R\ge 4/\delta$.
Moreover, let
$F_j=\chi _{\Omega _j}F$, $ j = 1,\dots, 5$, and
$u_j = \eabs \cdo ^{s_j} v_j$, $ j=1,2$.
Then the estimate
\begin{equation*}
\nm {T_ {F_j} (u_1, u_2)}{L^{q_0'}}
\lesssim  \nm {v_1}{L ^ {q_1} _{s_1}}  \nm {v_2}{L ^ {q_2} _{s_2}}
\end{equation*}
holds when:
\begin{enumerate}
\item $j=1,2$, for $\masfR (q)\le 1/q_0$, $s_0\leq s_1$, $s_0\leq s_2$ and
$$
s_0  \leq s_1 + s_2 - d \cdot \masfR (q),
$$
where the above inequality is strict when $ s_1 = d\cdot \masfR (q)$
or $  s_2 = d \cdot \masfR (q)$.

\vrum

\item $j=3$, for
$$
\begin{cases}
\masfR (q) \le \min(1/q_1,1/q_2) & \text {when} \quad  q_1, q_2 < 2 ,
\\[1 ex]
\masfR (q) \le 1/2_{\phantom 2}
& \text {when}  \quad q_1\geq 2 \quad \text{or} \quad q_2\geq 2,
\end{cases}
$$
and
$$
0 \leq s_1 + s_2\text ;
$$

\vrum

\item $j=4$ for
$
\masfR (q) \le \max(1/q_2,1/2), 
%
$
$$
0 \leq s_1 + s_2
\quad
\text{and}
\quad
s_0  \leq s_1 + s_2 - d\cdot \masfR (q),
$$
with $ 0< s_1 + s_2 $ when
$ s_0 =  - d \cdot \masfR (q)$;

\vrum

\item $j=5$, for
$
\masfR (q) \le \max(1/q_1,1/2), 
$
$$
0 \leq s_1 + s_2
\quad
\text{and}
\quad
s_0  \leq s_1 + s_2 - d\cdot \masfR (q),
$$
with $ 0< s_1 + s_2 $ when $ s_0 =  - d \cdot \masfR (q)$.
\end{enumerate}
\end{lemma}

\par

\begin{proof}
Let $r=1/\masfR (q)$.

\par

(1) The condition $\masfR (q)\le 1/q_0$ implies that $ r \geq q_0'$.
%
By Lemma \ref{intestimates} (1) it follows that
\begin{equation}\label{Fjest1}
\nm {F_1}{L^{\infty ,r}_2}<\infty
\end{equation}
when $s_0\leq s_1$ and 
$$
\begin{cases}
s _0\leq s_1 + s_2 - d/r, \quad &\text {for} \quad  s_2 \neq d/ r
\\[1 ex]
s _0< s_1,  \quad &\text {for} \quad  s_2 = d/ r.
\end{cases}
$$
Similarly, by Lemma \ref{intestimates} (2) it follows that
\begin{equation}\label{Fjest2}
\nm {F_2}{L^{\infty ,r}_2}<\infty
\end{equation}
when $s_0\leq s_2$ and 
$$
\begin{cases}
s_0 \leq s_1 + s_2 - d/r, \quad &\text {for} \quad  s_1 \neq d/ r
\\[1 ex]
s_0 < s_2, \quad &\text {for} \quad  s_1 = d / r.
\end{cases}
$$
This, together with Proposition \ref{prop1} (1) gives
$$
\nm {T_ {F_j} (u_1, u_2)}{L^{q_0'}}
\lesssim
 \nm {v_1}{L ^ {q_1} _{s_1}}  \nm {v_2}{L ^ {q_2} _{s_2}}, \quad j = 1,2.
$$

\par

(2) By Lemma \ref{intestimates} (3) we have
\begin{equation}\label{Fjest3}
\nm {F_3}{L^{r_0,\infty }_1}<\infty ,
\end{equation}
when  $s_1 + s_2\geq 0$ and $ r_0 \in [1,\infty]$.
In particular, if $r_0=r=1/\masfR(q)$ and $r \geq \min (2, \max (q_1, q_2))$, then
it follows from Proposition \ref{prop1} (2) and (3) that
\begin{equation*}
\nm {T_ {F_3} (u_1, u_2)}{L^{r}}
\leq
C \nm {v_1}{L ^ {q_1} _{s_1}}\nm {v_2}{L ^ {q_2} _{s_2}}.
\end{equation*}
This gives (2).

\par

Next consider $T_{F_4}$ and $T_{F_5}$.
By Lemma \ref{intestimates} (4) it follows that
\begin{equation}\label{Fjest4-5}
\nm {F_4}{L^{r,\infty}_1}<\infty \quad \mbox{and} \quad
\nm {F_5}{L^{r,\infty}_1}<\infty
\end{equation}
when
$$
\begin{cases}
s_0-s_1-s_2 + d / r \leq 0,
\quad & s_0 > -d / r
\\[1 ex]
s_1 + s_2 > 0, \quad & s_0 = - d / r
\\[1 ex]
s_1 + s_2 \geq 0, \quad & s _0< - d / r.
\end{cases}
$$
If $s_0 > -d / r$ and  $s_0-s_1-s_2 + d / r \leq 0$, then $ s_1 + s_2 > 0$.
Therefore \eqref{Fjest4-5} holds when
$$
0 \leq s_1 + s_2
$$
and
$$
s_0  \leq s_1 + s_2 - d / r,
$$
with $ 0< s_1 + s_2 $ when $ s_0 =  - d / r $.
Hence Proposition \ref{prop1} (3) gives
\begin{equation*}
\nm {T_ {F_4} (u_1, u_2)}{L^{q_0'}} \lesssim \nm {v_1}{L ^ {q_1} _{s_1}}\nm {v_2}{L ^ {q_2} _{s_2}}
\end{equation*}
for $r \geq \min (2, q_2)$, and (3) follows.
%

\par

Finally, by Proposition \ref{prop1} (2)
we get that
\begin{equation*}
\nm {T_ {F_5} (u_1, u_2)}{L^{q_0}} \lesssim \nm {v_1}{L ^ {q_1} _{s_1}}
\nm {v_2}{L ^ {q_2} _{s_2}}
\end{equation*}
when
$r \geq \min (2, q_1)$. This gives (4), and the proof is complete.
%
%
%
\end{proof}

\par

In the following lemma we give another view to Lemma \ref{estimatesTFj},
which will be used for the proof of Theorem \ref{together}.

\par

\begin{lemma} \label{estimatesT_F}
Let $F$, $F_j$ and $u_j$ be the same as in Lemma \ref{estimatesTFj}. Furthermore,
assume that \eqref{lastineq1} and \eqref{lastineq2}
hold, with strict inequality in the last inequality in \eqref{lastineq2}
when $ s_1, s_2 $ or $-s_0$ is equal to $ d \cdot \masfR (q)$. Then
\begin{equation*}
\nm {T_ {F_j} (u_1, u_2)}{L^{q_0}}
\lesssim  \nm {v_1}{L ^ {q_1} _{s_1}}\nm {v_2}{L ^ {q_2} _{s_2}}
\end{equation*}
holds for every $j \in \{1,\dots, 5 \}$.

\par

Furthermore, if the conditions in \eqref{lastineq1} and \eqref{lastineq2} are violated,
then at least one of the relations in {\rm{(1)-(5)}}
in Lemma \ref{estimatesTFj} is violated.
\end{lemma}

\par

We have now the following result which is needed for the proof of
Theorem \ref{together}.

\par

\begin{prop}\label{proptogether}
Let $s_j \in \mathbf R$, $q_j \in
[1,\infty] $, $j=0,1,2$ and let $\masfR (q)$ be as in \eqref{Rqfunctional}.
Also assume that \eqref{lastineq1} and \eqref{lastineq2} 
hold, with strict inequality in the last inequality in \eqref{lastineq2}
when $\masfR (q)>0$ and $s_j=d\cdot  \masfR (q) $
for some $j=0,1,2$. Then  the map $(v_1,v_2)\mapsto v_1*v_2$
on $C_0^\infty (\rr d)$ extends uniquely to a continuous map from
$ L^{q_1} _{s_1}(\rr d)
\times  L^{q_2} _{s_2}(\rr d)$ to $L^{q_0'} _{-s_0}(\rr d)$.
\end{prop}

%
%
%

\begin{proof}
First we note that \eqref{lastineq1} is not fulfilled when all $q_j\ge 2$ and at least one
of them is strictly larger than $2$. The similar fact is true if the condition \eqref{lastineq1}
is replaced by
\begin{equation}\tag*{(\ref{lastineq1})$'$}
\masfR (q) \le \mathsf H (q),
\end{equation}
where $\mathsf H(q)=\min (q_0^{-1}, q_1^{-1}, q_2^{-1})$ when $q_j\le 2$,
$\mathsf H(q)=\max (q_0^{-1}, q_1^{-1}, q_2^{-1})$ when $q_j\ge 2$, $j=0,1,2$, and
$\mathsf H(q)=2^{-1}$ otherwise. Hence, we may replace the condition \eqref{lastineq1}
by \eqref{lastineq1}$'$ when proving the proposition.

\par

First we assume that
\begin{equation}\tag*{(\ref{lastineq1})$''$}
\masfR (q)\le \frac 1{q_0}\quad \text{and}\quad
\masfR (q)\le \max \left (  \frac 12, \min \left ( \frac 1{q_1},\frac 1{q_2}  \right ) \right ),
\end{equation}
and that \eqref{lastineq2} 
holds and  $v_j \in L^{q_j} _{s_j}$, $ j = 1,2$.
We express $ v_1* v_2$ in terms of $T_F $ given by \eqref{TFmap} and
$F$ given by \eqref{Fdef} as follows. Let $\Omega_j$,  $j=1,\dots ,5$,
be the same as in \eqref{theomegasets}
after $\Omega _2$ has been modified into
$$
\Omega _2 = \sets{ (\xi, \eta) \in \rr {2d}}{\eabs {\xi -\eta }<
\delta \eabs \xi }\setminus \Omega _1.
$$
Then $\cup \Omega _j=\rr {2d}$, $\Omega _j\cap
\Omega _k$ has Lebesgue measure zero when  $j\neq k$, and
\begin{multline*}
(v_1*v_2) (\xi) \eabs \xi ^s  = \int F (\xi, \eta ) u_1 (\xi - \eta) u_2 (\eta) d\eta
=T_F(u_1,u_2)
\\[1ex]
=T_{F_1}(u_1,u_2)+\cdots + T_{F_5}(u_1,u_2)
\end{multline*}
where $ u_j (\cdo ) = \eabs \cdo  ^{s_j} v_j $, $ j=1,2$,
and $ F_j = \chi_{\Omega _j}F$, $j = 1,\dots, 5.$

\par

Now, Lemma \ref{estimatesT_F} implies that the $L^{q_0'}$ norm of each
of the terms $ T_{F_j}$, $ j = 1,\dots, 5$
is bounded by $C \nm {v_1}{L ^ {q_1} _{s_1}}\nm {v_2}{L ^ {q_2} _{s_2}} $
for some positive constant $C$
which is independent of  $v_1 \in L^{q_1}_{s_1}(\rr d)$
and $v_2\in L^{q_2}_{s_2}(\rr d)$.

\par

Hence, $v_1 *v_2 \in L^{q_0'} _{-s_0}$ when \eqref{lastineq1}$''$ holds.
By duality, the same conclusion holds when  the roles for $q_j$, $j=0,1,2$ have
been interchanged. By straight forward computations it follows that \eqref{lastineq1}$'$
is fulfilled if and only if \eqref{lastineq1}$''$ or one of the dual cases of \eqref{lastineq1}$''$
are fulfilled. This gives the result.
\end{proof}

\par

\begin{proof}[Proof of Theorem \ref{together}]
The assertion (1) follows by letting $v_j=\widehat f_j$ in Proposition \ref{proptogether}.

\par

In order to prove (2), we assume that $f_j\in (\FL _{s_j}^{q_j})_{loc}$ and let
$\phi \in C_0^\infty (X)$. Then we choose $\phi _1=\phi $ and $\phi _2 \in
C_0^\infty (X)$ such that $\phi _2 =1$ on $\supp \phi$. Since $\phi _jf_j
\in \FL _{s_j}^{q_j}$, the right-hand side of
$$
f_1f_2 \phi  = (f_1\phi _1)(f_2 \phi _2)
$$
is well-defined, and defines an element in $ \FL _{s_0}^{q_0}$, in view of (1).
This gives (2).

\par

When proving (3) we first consider the case when $p_j,q_j<\infty$ for $j=1,2$. Then
$\mathscr S$ is dense in $M^{p_j,q_j}_{s_j,t_j}$ for $j=1,2$. Since $M^{p,q}_{s,t}$
decreases with $t$, and the map
$f\mapsto \eabs \cdo ^{t_0}f$ is a bijection from $M^{p,q}_{s,t+t_0}$ to $M^{p,q}_{s,t}$,
for every choices of $p,q\in [1,\infty]$ and $s,t,t_0\in \mathbf R$, it follows that we may assume
that $t_j=0$, $j=0,1,2$.

\par

We have
\begin{multline}\label{STFTproducts}
(V_{\phi}(f_1 f_2))(x,\xi )
=(2\pi ) ^{-d/2}
\big ( (V_{\phi _1}f_1)(x,\cdo )*(V_{\phi _2}f_2)(x,\cdo )\big )(\xi ),
\\[1ex]
\phi =\phi _1 \phi _2,\ \phi _j,f_j\in \mathscr S(\rr d),\ j=1,2,
\end{multline}
which follows by straight-forward application of Fourier's inversion formula.
Here the convolutions between the factors $(V_{\phi _j}f_j)(x,\xi )$, where $j=1,2$ should
be taken over the $\xi$ variable only.

\par

By applying the $L^{p_0}$ norm with respect to the $x$ variables
and using H{\"o}lder's inequality we get
$$
\nm {V_{\phi}(f_1f_2))(\cdo ,\xi )}{L^{p_0}} \le (2\pi ) ^{-d/2} (v_1*v_2)(\xi),
$$
where $v_j =  \nm {V_{\phi _j}f_j)(\cdo ,\eta )}{L^{p_j}} $. Hence by applying
the $L^{q_0}_{s_0}$ norm on the latter inequality and using Proposition
\ref{proptogether} we get
$$
\nm {f_1f_2}{M^{p_0,q_0}_{s_0,0}}\lesssim \nm {v_1}{L^{q_1}_{s_1}}\nm {v_2}{L^{q_2}_{s_2}}
\asymp \nm {f_1}{M^{p_1,q_1}_{s_1,0}}\nm {f_2}{M^{p_2,q_2}_{s_2,0}},
$$
and (3) follows in this case, since $\mathscr S$ is dense in $M^{p_j,q_j}_{s_j,0}$ for $j=1,2$.

\par

For general $p_j$ and $q_j$, (3) follows from the latter inequality
and Hahn-Banach's theorem.

\par

Finally, by interchanging the order of integration, (4) follows by similar
arguments as in the proof of (3). The proof is complete.
\end{proof}

Finally, we prove Theorem \ref{microlocal-functions-some-more}.

\begin{proof}[Proof of Theorem \ref{microlocal-functions-some-more}]
Again we only prove the result for $1<q<\infty$, leaving small
modifications when $q\in \{ 1,\infty \}$ to the reader.

\par

Assume that $(x_0,\xi _0)\notin
\WF _{\mathscr F L^{q_{1}} _{s_1}} (f_1)\cup \WF _{\mathscr F L^{q_2} _{s_2}} (f_2)$.
It is no restriction to assume that $f_j$ has compact
support and $\xi _0\notin \Sigma_{\mathscr FL^{q_j} _{s_j}}(f_j)$, $ j
=1,2$. Then $|f_j|_{\mathscr FL^{q_j,\Gamma}_{s_j}}<\infty$ for some
conic neighborhood $\Gamma$ of $\xi _0$. Furthermore, for some $\delta \in (0,1)$
and open cone $ \Gamma _1$ of $\xi _0$ such that $\overline  \Gamma
_1\subseteq \Gamma $ we have $\xi - \eta \in \Gamma$ when $\xi \in
\Gamma _1$ and $|\eta | < \delta |\xi |$.

Let $\Omega _1$ and $\Omega _2$ be the same as in \eqref{theomegasets},
$\Omega
_0=\complement (\Omega _1\cup \Omega _2)$,
and let
$$
J_k(\xi ) =  \eabs \xi ^s \int _{(\xi ,\eta )\in \Omega _k}  |\widehat
f_1(\xi -\eta )\widehat f_2(\eta )|\, d\eta
$$
for $k=0,1,2$.

%
%
%

\par

We have $C^{-1}\eabs \xi \le \eabs {\xi -\eta}\le C\eabs \xi$ when
$(\xi ,\eta )\in \Omega _1$. This gives 
\begin{multline} \label{J1estC}
J_1(\xi ) =  \eabs \xi ^{s} \int_{\Omega _1}
|\widehat f_1 (\xi - \eta)|\, |\widehat f_2 (\eta)| \, d\eta
\\[1ex]
\leq  \int_{\Omega _1} \eabs \xi ^{s} \langle \xi - \eta \rangle  ^{-s_2}
 \eabs \eta ^{-s_1}
|\chi_\Gamma \widehat f_1 (\xi - \eta)| \langle \xi - \eta \rangle ^{s_2} \, |\widehat f_2 (\eta)|
 \eabs \eta ^{s_1} \, d\eta.
\end{multline}

Now, Lemma \ref{intestimates} (1) implies that 
\begin{equation}\label{J1estno2}
\nm {J_1}{L^q(\Gamma _1)} \le C \nm {f_1}{\mathscr FL^{q_1, \Gamma} _{s_1}}
|f_2 |_{\mathscr FL^{q_2 }_{s_2}}.
\end{equation}

Similarly,
\begin{equation}\label{J1estno3}
\nm {J_2}{L^q(\Gamma _1)} \le C \nm {f_1}{\mathscr FL^{q_1} _{s_1}}
|f_2 |_{\mathscr FL^{q_2, \Gamma }_{s_2}}.
\end{equation}

\par

Consider next $\Omega _0=\complement (\Omega _1\cup \Omega _2)$. Then
$$
\sets {(\xi ,\eta )\in \Omega _0}{\xi \in \Gamma _1}\subseteq \Omega
_3\cup \Omega _4\cup \Omega _5,
$$
where $ \Omega_j $, $ j=3,4,5 $ are the same as in \eqref{theomegasets},
and
$$
J_0\le T_{F_3}(u_1,u_2)+T_{F_4}(u_1,u_2)+T_{F_5}(u_1,u_2),
$$
where $u_j(\xi )=\widehat f_j(\xi )\eabs \xi ^{s_j}$,
$F_j$ are the same as in Lemma \ref{intestimates}, and
$T_{F_j} $ are  the same as in Lemma \ref{estimatesTFj}, $j=3,4,5$. Hence it
suffices to prove that
$$
\nm {T_{F_j}(u_1,u_2)}{L^q}\le C\nm {f_1}{\mathscr
FL^{q_1} _{s_1}}\nm {f_2}{\mathscr FL^{q_2} _{s_2}}, \;\;\; j=3,4,5.
$$

These estimates follow from Lemma  \ref{estimatesT_F}
which completes the proof.
\end{proof}

\end{document}